\documentclass[11pt]{article}

\usepackage[margin=1in]{geometry}
\usepackage{amsmath,amssymb,amsthm,booktabs,microtype,xcolor}
\usepackage[T1]{fontenc}
\usepackage{lmodern}
\definecolor{linkblue}{RGB}{0,70,140}
\usepackage[
  colorlinks=true,
  linkcolor=linkblue,
  citecolor=linkblue,
  urlcolor=linkblue,
  pdfborder={0 0 0},
  bookmarksopen=true
]{hyperref}
\allowdisplaybreaks

\newtheorem{theorem}{Theorem}
\newtheorem{lemma}[theorem]{Lemma}

\newcommand{\Z}{\mathbb{Z}}
\newcommand{\thmref}[1]{Theorem~\ref{#1}}
\newcommand{\lemref}[1]{Lemma~\ref{#1}}

\title{Positive Lower Density for Hofstadter's $ab-1$ Problem}
\author{Samuel Korsky}
\date{July 20, 2026}

\begin{document}

\maketitle

\begin{abstract}
\noindent
Let $A$ be the smallest set of positive integers containing $2$ and $3$ such that $ab-1\in A$ whenever $a,b\in A$ are distinct. We prove that $A$ has positive lower density, answering a problem of Erd\H{o}s attributed to Hofstadter.
\end{abstract}

\section{Introduction}

Let $A$ be the smallest set of positive integers containing $2$ and $3$ and closed under $ab-1$ for distinct $a,b\in A$. Erd\H{o}s, attributing the question to Hofstadter, asked whether $A$ has positive lower density \cite{Erdos1977,ErdosGraham1980}; this is Erd\H{o}s Problem 424 and OEIS A005244 \cite{Bloom424,OEISoriginal}.

\begin{theorem}\label{thm:main}
There is a constant $c>0$ such that
\[
|A\cap[1,x]|\geq cx
\]
for all sufficiently large $x$.
\end{theorem}

The proof constructs many distinct affine maps having the same slope. Compositions of the maps $T_a(x)=ax-1$ are encoded by paths through a finite partition of an interval. The transition probabilities are chosen so that the probability of any return to a fixed state is exactly the reciprocal of the slope of the corresponding affine map. By switching among four assignments of multipliers, we keep the exponents of $2,3,5,$ and $7$ in the slope close to fixed proportions. Consequently, every such return has slope $q^m$ for one fixed integer $q$. The arithmetic renewal theorem then shows that, for infinitely many values of $m$, there are on the order of $q^m$ distinct affine maps with slope $q^m$.

The restriction to distinct inputs is handled by using the multipliers
\[
2,3,5,9,14\in A
\]
and evaluating the resulting maps at $17\in A$. Indeed,
\[
5=2\cdot3-1,\qquad 9=2\cdot5-1,\qquad
14=3\cdot5-1,\qquad 17=2\cdot9-1,
\]
with distinct inputs throughout. Since every multiplier is less than $17$ and $T_a(x)>x$ for $a\geq2$ and $x>1$, every later operation at the seed $17$ also has distinct inputs.

Affine orbit sets and recursively generated integer sets were studied by
Klarner and Rado \cite{KlarnerRado1974}; Klarner subsequently developed a
zero-density algorithm, freeness criteria, and finite-automaton descriptions
for related affine closures
\cite{Klarner1981,Klarner1982,Klarner1988}, and Lagarias surveys this line of
work \cite{Lagarias2016}. A particularly close predecessor is the work of
Shamazov and Talambutsa \cite{ShamazovTalambutsa2026}, whose lower-bound
arguments count many distinct affine compositions having a common slope,
using freeness and, for their positive-density result, an exact covering
hypothesis. Inverse-interval ping-pong arguments for proving freeness of
affine semigroups also appear in Kolpakov and Talambutsa
\cite{KolpakovTalambutsa2022}. The semigroup used here is not free -- for
example,
\[
T_9\circ T_2=T_3\circ T_2\circ T_3,
\]
so neither global freeness nor free-semigroup permutation counting directly
provides the injectivity needed here; instead, the interval coding below
isolates an injective graph-directed family of paths. 

At a broader
methodological level, finite-state inverse-branch constructions and
place-dependent transition probabilities have precedents in graph-directed
and iterated-function systems
\cite{MauldinWilliams1988,BarnsleyEtAl1988}. The feedback rule used to select
a favorable drift direction is related in spirit to Blackwell approachability
and to MaxWeight and state-dependent Foster--Lyapunov stability methods
\cite{Blackwell1956,TassiulasEphremides1992,YukselMeyn2013}. Finally, the
renewal input is classical \cite{Feller}, and renewal methods have also been
used in graph-directed counting problems and in the density theory of
expanding affine orbit systems \cite{HamblyNyberg2003,MiaoXu2026}. The
distinguishing feature of the present proof is the combination of these
ideas: an injective graph-directed sublanguage inside a nonfree affine
semigroup, feedback control of prime-exponent imbalances forcing slopes
$q^m$, exact reciprocal-slope return probabilities, and renewal counting of
the resulting maps.

\section{A Finite Interval Coding}

Put
\[
T_a(x)=ax-1,
\qquad
h_a(x)=T_a^{-1}(x)=\frac{x+1}{a}.
\]
Let $K_i=[e_i,e_{i+1})$, where
\[
\begin{aligned}
(e_1,\ldots,e_{20})=\biggl(
&\frac19,\frac{10}{81},\frac17,\frac15,\frac29,\frac27,
\frac13,\frac{10}{27},\frac25,\frac{11}{27},\frac37,\frac{37}{81},\frac59,\frac35,\frac23,
\frac{19}{27},\frac{59}{81},\frac45,\frac{23}{27},1
\biggr).
\end{aligned}
\]
Write $K_r\text{--}K_s=K_r\cup\cdots\cup K_s$. The following are exact interval identities.

\begin{center}
\small
\begin{tabular}{ccl@{\qquad}ccl}
\toprule
state & multiplier & image & state & multiplier & image \\
\midrule
$K_1$ & $14$ & $K_{13}\text{--}K_{16}$ & $K_{11}$ & $3$ & $K_6\text{--}K_7$ \\
$K_2$ & $14$ & $K_{17}\text{--}K_{19}$ & $K_{12}$ & $3$ & $K_8\text{--}K_{14}$ \\
$K_2$ & $9$ & $K_1\text{--}K_5$ & $K_{13}$ & $3$ & $K_{15}\text{--}K_{17}$ \\
$K_3$ & $9$ & $K_6\text{--}K_{17}$ & $K_{13}$ & $2$ & $K_1\text{--}K_3$ \\
$K_4$ & $9$ & $K_{18}\text{--}K_{19}$ & $K_{14}$ & $3$ & $K_{18}\text{--}K_{19}$ \\
$K_5$ & $5$ & $K_1\text{--}K_{10}$ & $K_{14}$ & $2$ & $K_4\text{--}K_6$ \\
$K_6$ & $5$ & $K_{11}\text{--}K_{14}$ & $K_{15}$ & $2$ & $K_7\text{--}K_9$ \\
$K_7$ & $5$ & $K_{15}\text{--}K_{18}$ & $K_{16}$ & $2$ & $K_{10}\text{--}K_{11}$ \\
$K_8$ & $5$ & $K_{19}$ & $K_{17}$ & $2$ & $K_{12}\text{--}K_{13}$ \\
$K_8$ & $3$ & $K_1\text{--}K_3$ & $K_{18}$ & $2$ & $K_{14}\text{--}K_{15}$ \\
$K_9$ & $3$ & $K_4$ & $K_{19}$ & $2$ & $K_{16}\text{--}K_{19}$ \\
$K_{10}$ & $3$ & $K_5$ & & & \\
\bottomrule
\end{tabular}
\end{center}

At $K_2$ we may use $14$ or $9$, and at $K_8$ we may use $5$ or $3$. At $K_{13}$ and $K_{14}$ we make the same choice, using either $3$ at both states or $2$ at both states. We consider the following four assignments:
\[
\alpha_0=(14,5,3),\quad
\alpha_1=(9,5,3),\quad
\alpha_2=(14,3,3),\quad
\alpha_3=(14,5,2),
\]
where the entries record these three choices. Let $a_i^{(\alpha)}$ be the multiplier prescribed by $\alpha$ at $K_i$, and define
\[
p_{ij}^{(\alpha)}=\frac{|K_j|}{a_i^{(\alpha)}|K_i|}
\]
when $K_j\subseteq T_{a_i^{(\alpha)}}(K_i)$. The intervals $h_{a_i^{(\alpha)}}(K_j)$ partition $K_i$, so $P_\alpha=(p_{ij}^{(\alpha)})$ is stochastic. In each assignment every state can reach $K_{19}$, and $K_{19}$ can reach every state; hence each $P_\alpha$ is irreducible.

For any allowed path $\gamma=(i_0,\ldots,i_r)$ with successive multipliers $a_0,\ldots,a_{r-1}$,
\begin{equation}\label{eq:telescope}
\mathbb P(\gamma)
=\prod_{t=0}^{r-1}\frac{|K_{i_{t+1}}|}{a_t|K_{i_t}|}
=\frac{|K_{i_r}|}{|K_{i_0}|}\cdot\frac1{a_0\cdots a_{r-1}}.
\end{equation}
This identity remains valid when $\alpha$ changes along the path.

Let $\nu=(\nu_2,\nu_3,\nu_5,\nu_7)$ be the prime-exponent vector of the accumulated slope, and put
\[
H(\nu)=(\nu_2-31\nu_7,\ \nu_3-26\nu_7,\ \nu_5-11\nu_7)\in\Z^3.
\]
Write $I_t$ for the interval state after $t$ steps, let $\nu(t)$ be the prime-exponent vector of the slope accumulated during those steps, and put
\[
H_t=H(\nu(t)).
\]
The coefficients $31,26,11$ are chosen so that the four stationary mean increments below have the origin in the interior of their convex hull. Put
\[
q=2^{31}3^{26}5^{11}7.
\]
Then $H(\nu)=0$ precisely when
\[
\nu=m(31,26,11,1)
\]
for some $m\geq0$, in which case the corresponding slope is $q^m$.

For fixed $\alpha$, let $d_\alpha$ be the stationary mean increment of $H$. Solving $\pi_\alpha P_\alpha=\pi_\alpha$ over $\mathbb Q$ and summing the one-step increments against $\pi_\alpha$ gives
\[
\begin{array}{c|c}
\alpha & d_\alpha \\ \hline
\alpha_0 & \dfrac1{209179}(82502,60472,22315)\\[1ex]
\alpha_1 & \dfrac1{9240032}(4122816,3176276,1187745)\\[1ex]
\alpha_2 & \dfrac1{154297}(-17374,2952,-12847)\\[1ex]
\alpha_3 & \dfrac1{938829}(-410774,-384064,-82315).
\end{array}
\]
The origin lies in the interior of
\[
\operatorname{conv}\{d_{\alpha_0},d_{\alpha_1},d_{\alpha_2},d_{\alpha_3}\}.
\]

\section{Recurrence}

\begin{lemma}\label{lem:recurrence}
There is a deterministic way to choose among $\alpha_0,\ldots,\alpha_3$ such that the interval state together with the imbalance vector has a positive recurrent state.
\end{lemma}

\begin{proof}
For fixed $\alpha$, let $\xi_\alpha(i)$ be the increment of $H_t$ on leaving $K_i$. Since $P_\alpha$ is irreducible on a finite state space and $\xi_\alpha-d_\alpha$ has stationary mean zero, the Markov-chain Poisson equation
\[
g_\alpha-P_\alpha g_\alpha=\xi_\alpha-d_\alpha
\]
has a bounded solution, coordinatewise. Hence there is a constant $C_0$, independent of $\alpha$, the initial state $i$, and $N\geq1$, such that
\begin{equation}\label{eq:averaging}
\left|
\mathbb E_i\left[\sum_{t=0}^{N-1}\xi_\alpha(I_t)\right]-Nd_\alpha
\right|\leq C_0.
\end{equation}

Since the origin lies in the interior of the convex hull of the four drift vectors, there is $\delta>0$ such that for every unit vector $v\in\mathbb R^3$,
\[
\min_\alpha\langle v,d_\alpha\rangle\leq-\delta.
\]
At each time $kN$, choose $\alpha_k\in\{\alpha_0,\alpha_1,\alpha_2,\alpha_3\}$ and use $P_{\alpha_k}$ for the next $N$ steps. If $H_{kN}=h\neq0$, choose $\alpha_k$ to minimize
\[
\left\langle \frac{h}{|h|},d_\alpha\right\rangle.
\]
Fix a deterministic tie-breaking convention, and take $\alpha_k=\alpha_0$ when $h=0$. Then
\[
X_k=(I_{kN},H_{kN})
\]
is a time-homogeneous Markov chain.

Put
\[
M_0=\max_{\alpha,i}|\xi_\alpha(i)|.
\]
Combining \eqref{eq:averaging} with
\[
|h+z|-|h|\leq \left\langle\frac h{|h|},z\right\rangle+\frac{|z|^2}{2|h|}
\]
gives
\[
\mathbb E\left[
|H_{(k+1)N}|-|H_{kN}|
\,\middle|\,
X_k=(i,h)
\right]
\leq-N\delta+C_0+\frac{M_0^2N^2}{2|h|}.
\]
First choose $N$ with $N\delta>C_0+2$, and then choose $R$ so that the right-hand side is at most $-1$ whenever $|h|>R$.

Let
\[
\mathcal C=\{(i,h):1\leq i\leq19,\ h\in\Z^3,\ |h|\leq R\}
\]
and define
\[
\tau_{\mathcal C}=\inf\{k\geq0:X_k\in\mathcal C\},
\qquad
\tau_{\mathcal C}^+=\inf\{k\geq1:X_k\in\mathcal C\}.
\]
The drift estimate and optional stopping give
\[
\mathbb E_x\left[\tau_{\mathcal C}\right]\leq |h|
\]
for $x=(i,h)\notin\mathcal C$. If $x\in\mathcal C$, then $|H_N|\leq R+M_0N$, and the Markov property gives
\[
\begin{aligned}
\mathbb E_x\left[\tau_{\mathcal C}^+\right]
&=1+\mathbb E_x\left[
\mathbf 1_{\{X_1\notin\mathcal C\}}
\mathbb E_{X_1}\left[\tau_{\mathcal C}\right]
\right]\\
&\leq1+\mathbb E_x\left[
\mathbf 1_{\{X_1\notin\mathcal C\}}|H_N|
\right]\\
&\leq1+R+M_0N.
\end{aligned}
\]
The chain induced by successive visits to $\mathcal C$ has finite state space. Choose $s$ in one of its recurrent classes. Its return time in the induced chain has finite mean, and the uniform bound above on the expected time between successive visits to $\mathcal C$ implies that its return time in $(X_k)$ also has finite mean. Thus $s$ is positive recurrent.
\end{proof}

\section{Return Paths and Renewal}

Fix the state $s=(i_\ast,h_\ast)$ supplied by \lemref{lem:recurrence}, and let
\[
\tau=\inf\{k\geq1:X_k=s\}.
\]
For each realization of this first return, record the interval states during the $N\tau$ underlying steps. Let $\mathcal R$ be the collection of all finite interval paths obtained in this way. For $\gamma=(i_0,\ldots,i_r)\in\mathcal R$, write
\[
G_\gamma=T_{a_{r-1}}\circ\cdots\circ T_{a_0},
\qquad
J_\gamma=G_\gamma^{-1}(K_{i_r})\subseteq K_{i_0}.
\]

\begin{lemma}\label{lem:injective}
Distinct finite concatenations of paths in $\mathcal R$ give distinct affine maps.
\end{lemma}

\begin{proof}
No member of $\mathcal R$ is a proper initial segment of another, by first return. Thus two distinct paths in $\mathcal R$ have a first step at which they enter different intervals. Immediately before that transition they have the same interval history; by the deterministic choice of $\alpha$, they also have the same accumulated imbalance and use the same multiplier $a$. If their next intervals are $K_j$ and $K_{j'}$, with $j\neq j'$, then
\[
h_a(K_j)\cap h_a(K_{j'})=\varnothing.
\]
The inverse-image intervals associated with the two complete paths lie in the pullbacks of these disjoint sets through their common initial composition, and are therefore disjoint. Since each path begins and ends at $K_{i_\ast}$, equality of the corresponding affine maps would force equality of their inverse images of $K_{i_\ast}$, a contradiction.

Returns to $s$ determine the decomposition into first-return paths uniquely. After deleting the common initial paths from two distinct concatenations, either both have a path remaining or only one does. In the first case, the inverse-image intervals of all continuations are contained in the disjoint intervals associated with the first differing paths. In the second case, the two affine maps have different slopes. Thus the maps are distinct.
\end{proof}

For $\gamma\in\mathcal R$, let $Q_\gamma$ be the slope of $G_\gamma$. Since a return to $s$ returns to the same interval, \eqref{eq:telescope} gives
\[
\mathbb P(\gamma)=Q_\gamma^{-1}.
\]
It also returns to the same imbalance vector, so the net prime-exponent vector is $m(31,26,11,1)$ for a unique $m\geq1$; hence $Q_\gamma=q^m$.

Let $b_m$ be the number of first-return paths with slope $q^m$. Since $s$ is positive recurrent, $\tau<\infty$ almost surely. The events corresponding to the paths in $\mathcal R$ are disjoint and exhaust the first return. Therefore
\begin{equation}\label{eq:return-law}
\sum_{m\geq1}b_mq^{-m}=1.
\end{equation}
Moreover, if such a path has length $L$, then each multiplier contributes one or two prime factors, counted with multiplicity, whereas $q^m$ has $69m$ prime factors. Thus
\[
L\leq69m\leq2L.
\]
It follows that each $b_m$ is finite. Let $M$ be the exponent of the random first-return path. Since $L=N\tau$ and $s$ is positive recurrent,
\begin{equation}\label{eq:finite-mean}
\begin{aligned}
\sum_{m\geq1}mb_mq^{-m}
&=\mathbb E_s\left[M\right]\\
&\leq\frac2{69}\cdot\mathbb E_s\left[L\right]\\
&=\frac{2N}{69}\cdot\mathbb E_s\left[\tau\right]
<\infty.
\end{aligned}
\end{equation}

\begin{proof}[Proof of \thmref{thm:main}]
Set $p_m=b_mq^{-m}$. By \eqref{eq:return-law} and \eqref{eq:finite-mean}, $(p_m)$ is a probability distribution on the positive integers with finite mean
\[
\mu=\mathbb E_s\left[M\right]=\sum_{m\geq1}mp_m.
\]
By the strong Markov property, the exponents of successive returns to $s$ are independent and have common distribution $(p_m)$. Let $c_n$ be the number of ordered concatenations of return paths whose total slope is $q^n$, including the empty concatenation for $n=0$. Each has probability $q^{-n}$, and therefore
\[
\frac{c_n}{q^n}=\sum_{k\geq0}p^{*k}(n).
\]
If $d$ is the greatest common divisor of the support of $(p_m)$, the arithmetic renewal theorem gives
\[
\frac{c_{rd}}{q^{rd}}\longrightarrow\frac d\mu>0
\]
as $r\to\infty$ \cite[Chapter XI, Section 1, p.~358]{Feller}. Hence, for some $\kappa>0$,
\[
c_{rd}\geq\kappa q^{rd}
\]
for all sufficiently large $r$.

An immediate induction shows that every nonempty composition has the form
\[
G(x)=Qx-C,
\qquad 0<C<Q.
\]
By \lemref{lem:injective}, the $c_{rd}$ concatenations give distinct maps of slope $Q=q^{rd}$, and hence distinct values at $17$. As observed in the introduction, every operation in computing $G(17)$ uses distinct elements of $A$, so $G(17)\in A$. Also
\[
16Q<G(17)<17Q.
\]
Consequently
\[
|A\cap[1,17q^{rd}]|\geq\kappa q^{rd}.
\]
For arbitrary sufficiently large $x$, choose $r$ with
$17q^{rd}\leq x<17q^{(r+1)d}$. Then
\[
\frac{|A\cap[1,x]|}{x}\geq\frac{\kappa}{17q^d}>0.
\]
\end{proof}

\section*{Acknowledgments}

The author thanks Thomas Bloom for helpful advice on the exposition, and Boris Alexeev for formalizing the proof in Lean \cite{AlexeevCodexErdos424}. GPT-5.6 Pro assisted in searching for and checking the finite interval and drift data. The author verified the argument and takes responsibility for the proof.

\end{document}